\documentclass[a4paper]{article}

\def \ms {\medskip}

\def \nd {\noindent}
\def \absolute {|}

\def \triangleq {:=}
\def \cf {{\cal F}}
\def \tt {{t\wedge \tau_R}}
\def \t {\tau}
\def \mn{\medskip\\ \noindent}
\def \rw {\rightarrow}

\usepackage{amsmath}
\usepackage{latexsym}
\usepackage{graphicx}
\usepackage{color}
\makeatletter
\def\@maketitle{\newpage
    \null
    \vskip .8truein
    \begin{center}%
     {\bf \@title \par}%
     \vskip 1.5em
     {\small
      \lineskip .5em
      \begin{tabular}[t]{c}\@author
      \end{tabular}\par}%
    \end{center}%
    \par
    \vskip .4truein}
\@addtoreset{equation}{section} \@addtoreset{theorem}{section}
\@addtoreset{lemma}{section} \@addtoreset{proposition}{section}
\@addtoreset{definition}{section} \@addtoreset{corollary}{section}
\@addtoreset{remark}{section}

\let\O=\Omega

\let\s=\sigma

\let\nn=\nonumber

\let\qq=\qquad

\let\ol=\overline
\newcommand{\re}{{{I\!\!R}}}
\newcommand{\na}{{{I\!\!N}}}

\def\R{{\bf R}}
\def\F{{\bf F}}
\let\a=\alpha

\def \no {\noindent}

\def \xt {X^{t,x}}
\def\qed{\hspace*{\fill} $\Box$\par\medskip}
\def \esp {[0,T]\times \R^N}

\def \be {\begin{equation}}
\def \ee {\end{equation}}
\def \ba{\begin{array}}
\def \ea{\end{array}}

\newtheorem{theorem}{Theorem}[section]

\newtheorem{definition}{Definition}[section]
\newtheorem{corollary}{Corollary}[section]
\newtheorem{remark}{Remark}[section]

\DeclareMathOperator{\argmax}{argmax}

\DeclareMathOperator{\Heav}{Heav}

\def \ed{\end{document}}
\def\proof{\list{}{\setlength{\leftmargin}{0pt}
                      \parskip=0pt\parsep=0pt\listparindent=2em
                      \itemindent=0pt}\item[]\futurelet\testchar\@maybe}

\def\@maybe{\ifx[\testchar \let\next\@Opt
          \else \let\next\@NoOpt \fi \next}
\def\@Opt[#1]{{\it Proof of #1.\ }}\def\@NoOpt{{\it Proof.\ }}

\begin{document}
\title{\Large \bf Regularity of Nash payoffs of Markovian nonzero-sum stochastic differential games.}
\author{{\Large Said Hamadene}\thanks{LMM, Le Mans Universit\'e, Avenue Olivier Messiaen, 72085 Le Mans, Cedex 9, France. \texttt{e-mail: hamadene@univ-lemans.fr}} {\Large \,\,\,and} {\Large Paola Mannucci}\thanks{Dipartimento di Matematica "Tullio Levi Civita",
Universit\`a degli Studi di Padova,
Via Trieste, 63, 35121, Padova, Italy. \texttt{e-mail: mannucci@math.unipd.it
}}}

\date{\today}
\maketitle
\vspace{0.5cm}

\noindent {\textbf{Abstract.}} In this paper we deal with the problem of existence of a smooth solution of the Hamilton-Jacobi-Bellman-Isaacs (HJBI for short) system of equations associated with nonzero-sum stochastic differential games. We consider the problem in unbounded domains either in the case of continuous generators or for discontinuous ones. In each case we show the existence of a smooth solution of the system. As a consequence, we show that the game has smooth Nash payoffs which are given by means of the solution of the HJBI system and the stochastic process which governs the dynamic of the controlled system.
\medskip

\medskip

\noindent{\textbf{MSC2010 subsject classification}}: 91A23 ; 49N10 ; 35R05 ; 49K20.

\section{Introduction}\label{intro}

\index{introduction}This article deals with a nonzero-sum stochastic
differential game (NZSDG for short) which we
describe hereafter. Let us consider a system, on which intervene two
players $\pi_1$ and $\pi_2$, whose dynamics is given by a solution
of a stochastic differential equation of the following form: \be
\label{eq:state}
dx_t^{u_1,u_2}=f(t,x_t^{u_1,u_2},u_{1t},u_{2t})dt+\sigma(t,x_t^{u_1,u_2})dB_t,\,t\leq
T \mbox{ and } x^{u_1,u_2}_0=x\in \re^N\ee where:

(i) $B:=(B_t)_{t\leq T}$ is a Brownian motion ;

(ii) $u_1:=(u_{1t})_{t\leq T}$ (resp. $u_2:=(u_{2t})_{t\leq T}$) is a
stochastic process with values in $U_1$ (resp. $U_2$) a compact metric
space and adapted w.r.t $(\cf_t)_{t\leq T}$, the completed natural
filtration of $B$. The process $u_1$ (resp. $u_2$) is the way by which the first
(resp. second) player $\pi_1$ (resp. $\pi_2$) acts on the system ;

(iii) $f(\cdot)$ and $\sigma(\cdot)$ are given functions.
\medskip

\noindent The system that one implies could be an asset in the
financial market, an economic unit, a factor in the economic or financial spheres, etc.  On the other hand, one can consider the differential game with more
than two players and this does not rise a major issue, the treatment
is the same.
\medskip

The conditional payoff of player $\pi_1$ (resp. $\pi_2$) from $t$ to $T$, when she implements $u_1$ (resp. $u_2$), is denoted $J_t^1(u_1,u_2)$ (resp. $J_t^2(u_1,u_2)$) and given by: for $i=1,2$, 
$$\ba{l}
J_t^i(u_1,u_2)=E[\int_t^Th_i(s,x^{u_1,u_2}_s,u_{1s},u_{2s})ds+g_i(x^{u_1,u_2}_T)|\cf_t].\ea
$$
The functions $h_1$, $h_2$ (resp. $g_1$, $g_2$) stand for the intantaneous (resp. terminal) payoffs of the players $\pi_1$, $\pi_2$, respectively. Then the problem of interest is to find a Nash equilibrium point for the game,
i.e., a pair of controls of the players $(u_1^*,u_2^*)$ such that
$$
J_0^1(u_1^*,u_2^*)\geq J_0^1(u_1,u_2^*) \mbox{ and }J_0^2(u_1^*,u_2^*)\geq
J_0^2(u_1^*,u_2^*) \mbox{ for any }u_1,u_2.
$$ On the other hand it is important to highlight the regularity properties of the conditional payoffs $J_t^i(u_1^*,u_2^*),\,\, t\leq T$, $i=1,2$, called conditional Nash payoffs of the game. The meaning of $(u_1^*,u_2^*)$ is that none of the players gains if she/he decides to deviate unilaterally.

In bounded domains this topic is already considered, e.g. in the monograph by Bensoussan-Frehse \cite{BeFr2}.
So far there are many papers which deal with nonzero-sum stochastic
differential games in a framework similar to ours, among which one can quote \cite{BeFr0, BeFr1, BeFr2, BoGh, Frie1, hlp, HaMu, hamu1, hamnote, PM, PM2}. They can be divided into three categories. In the first category one can group the works where the non-zerosum game is tackled by using probabilitic tools, namely backward stochastic differential equations (BSDE for short) \cite{hlp, HaMu, hamu1}; to solve the problem it is enough to solve its associated BSDE which is multi-dimensional with non-Lipschitz coefficient. However this solvability is not obvious and it is achieved only in the Markovian framework. The latter papers are related to various features of the data of the game, e.g., they are bounded  in \cite{hlp}  while this boudedness is partially removed in \cite{hamu1} and finally, in \cite{HaMu}, the authors consider the case when the coefficients of the multi-dimensional BSDE associated with the game  are discontinuous and the Nash point is of bang-bang type. In the second category one can gather the papers which use PDEs to tackle this non-zerosum differential game problem \cite{BeFr0, BeFr1, BeFr2, Frie1, PM, PM2}. Mainly  in those works, 
firstly the authors provide a regular solution for the Hamilton-Jacobi-Bellman system of equations associated with the game and then construct a Nash equilibrium point.   More precisely, in the papers \cite{BeFr0, BeFr1, BeFr2, Frie1}
the existence of Nash equilibria are proved under the assumption that the feedback is continuous, while in 
\cite{PM, PM2} the study is done loosing continuity of the feedback and hence of the Hamiltonians. Finally in the third category one can range papers, rather rare, where a mix of both of the previous methods are used, e.g., in \cite{BoGh, hamnote}. Note that in \cite{BoGh}, the controls are of relaxed type. 

The general case of path dependent process $(x_t^{u_1,u_2})_{t\leq
T}$ solution of \eqref{eq:state} is still open since, as pointed out previously, to
tackle this type of nonzero-sum SDG leads to deal with a multidimensional BSDE with
non Lipschitz coefficients and non markovian randomness,
for which there is a lack of result (see for instance \cite{FrDos}). 
 
The probabilistic approach can be described as: Let $H_i$, $i=1,2$, be the Hamiltonians associated with this game problem, i.e., for $i=1,2$ and $(t,x,u_1,u_2,p_1,p_2)\in [0,T]\times \re^N\times U_1\times U_2\times \re^{N+N}$,
\begin{equation}\label{Hamdef}
H_i(t,x,p_i,u_1,u_2):=p_i^\top f(t,x,u_1,u_2)+h_i(t,x,u_1,u_2)
\end{equation}
and assume that the following generalized Isaacs condition (GIC for short) is satisfied: 
\mn
\noindent (\textbf{A0}): There exist measurable functions $\bar u_1(t,x,p_1,p_2)$ and $\bar u_2(t,x,p_1,p_2)$ valued respectively in $U_1$ and $U_2$ such that for any $(t,x,p_1,p_2,u_1,u_2)$,
\begin{equation}\label{gicintro}\begin{array}{l}
H_1(t,x,p_1,\bar u_1(t,x,p_1,p_2),\bar u_2(t,x,p_1,p_2))\geq H_1(t,x,p_1,u_1,\bar u_2(t,x,p_1,p_2)) \\\mbox{ and } \\
H_2(t,x,p_2,\bar u_1(t,x,p_1,p_2),\bar u_2(t,x,p_1,p_2))\geq H_2(t,x,p_2,\bar u_1(t,x,p_1,p_2),u_2).\end{array}\end{equation}
This condition is the analogous of the Isaacs one in the framework of zero-sum differential games.

Next assume there exist adapted stochastic processes $(Y^1,Y^2,Z^1,Z^2)$, solution of the following system of two coupled BSDEs: for $i=1,2$ 
\begin{equation}\label{bsdeg}\left \{\begin{array}{l}
Y^i_t=g_i(x_T)+\int_t^TH_i(s,x_s,\sigma^{-1}(s,x_s)^\top Z^i_s,(\bar u_1,\bar u_2)(s,x_s,\sigma^{-1}(s,x_s)^\top Z^1_s,\sigma^{-1}(s,x_s)^\top Z^2_s))ds\\\qquad \qquad \qquad \qquad \qquad \qquad-\int_t^TZ^i_sdB_s,\,\,t\leq T,\end{array}\right. \end{equation}
where $(x_t)_{t\leq T}$ is the solution of \eqref{eq:state} without drift term (see \eqref{original equation of x} below) then \\$(u_1^*,u_2^*):=(\bar u_1(t,x_t,Z^{1,\sigma}_t, Z^{2,\sigma}_t),\bar u_2(t,x_t,Z^{1,\sigma}_t, Z^{2,\sigma}_t))_{t\leq T}$ (with $Z^{i,\sigma}_t=\sigma^{-1}(t,x_t)^\top Z^i_t$, $i=1,2$) is a Nash equilibrium point for the nonzero-sum differential game and $Y^i_t=J^i_t(u_1^*,u_2^*)$, $i=1,2$. Thus the problem turns into looking for a solution of the two-dimesional BSDE (\ref{bsdeg}) which is associated with the game problem. This point of view has been considered among others in \cite{hlp, hamu1}, where the existence of a Nash point for the game is shown under appropriate assumptions on the data of the problem. It must be said that the link between the processes $Y^i$ and $Z^i$ which allows for the construction of the Nash equilibrium point of the game is not very well understood. Mainly because there is a need of further regularity properties of the processes $Y^i$, $i=1,2$, which are not established yet.
\ms

As written before, the second approach uses partial differential equations tools (see e.g. \cite{BeFr1, Frie1, Frie,PM, PM2} and the references therein)  and mainly it turns into seeking a regular solution of the Hamilton-Jacobi-Bellman-Isaacs equations associated with this game problem, under various assumptions on the regularity of the feedbacks, which is the following: for $i=1,2$,
\begin{equation}\label{hjbedp}\!\left\{\begin{array}{l}
-\partial_t V^i(t,x)-\frac{1}{2}Tr[\sigma \sigma^\top(t,x)D^2_{xx}V^i(t,x)]
\\\\ \qquad 
-\nabla_xV^i(t,x).f(t,x,(\bar u_1,\bar u_2)(t,x,\nabla_xV^1(t,x),\nabla_xV^2(t,x)))\\\\\qquad-h_i(t,x,(\bar u_1,\bar u_2)(t,x,\nabla_xV^1(t,x),\nabla_xV^2(t,x)))=0, (t,x)\in R_T:=(0,T)\times \re^N\,;\\\\
V^i(T,x)=g^i(x), \mbox{ for }x\in \re^N.\end{array}\right.
\end{equation}
This system is the verification theorem of the NZSD game problem. Indeed if a regular solution of (\ref{hjbedp}) exists then by the use of It\^o-Krylov formula to $V^i(t,x_t)$ one obtains that the pair of controls $(u_1^*,u_2^*):=((\bar u_1,\bar u_2)(t,x_t, \nabla_xV^1(t,x_t),\nabla_xV^2(t,x_t)))_{t\leq T}$ is a Nash equilibrium point for the game and additionally $J_t^i(u_1^*,u_2^*)=V^i(t,x_t)$, $i=1,2$. This approach provides also more regular properties of the Nash payoffs of the game which we cannot obtain from the probabilistic one. This is helpful at least in: i) the understanding of the link between $J_t^i(u_1^*,u_2^*)$, $i=1,2$, and $(u_1^*$, $u_2^*)$ ; ii) the simulation process of either the Nash payoffs or the Nash equilibrium points where usually smoothness properties of the data are required. However, to the best of our knowlegde, system of equations (\ref{hjbedp}), with lack of regularity on $\overline u_i$, is studied only in the case when the domain $R_T$ is bounded. Therefore the main objective of this work is to deal with the same problem when $R_T$ is unbounded. 
Note that, for unbounded domains, a verification result for Markovian feedback controls can be obtained (see e.g. Theorem 8.5 of \cite{Do}).\\
This paper is organized as follows: In Section 2, we introduce precisely the nonzero-sum differential game which we will study later on. Section 3 is devoted to the study of the HJBI system (\ref{hjbedp}) associated with the NZSDG. We consider three different cases. In the first one we show that the system has a solution when the data of the problem are mainly continuous and  bounded. Then we treat the case when the data are bounded discontinuous and
finally we deal with the case when the Hamiltonian are discontinuous and the data have linear or polynomial growth. In Section 4, we study the connection of the solutions of the system with the NZSDG problem. We provide the Nash equilibrium point for the game and some regularity properties of its conditional Nash payoffs.

\section{Statement of the NZSDG problem}

Let $(\Omega, \F , P)$ be a probability space which carries a $N$-dimensional Brownian motion  $(B_t)_{t\leq T}$ whose completed natural filtration is $(\mathcal{F}_t= \sigma\{W_s, s\leq T\})_{t\leq T}$ and $\mathcal{P}$ is the $\sigma$-algebra on $[0,T]\times \Omega$ of
$\cf_t$-progressively measurable processes.
\medskip

\noindent Let $\sigma$ be a Borel  measurable function from $[0, T] \times \re^N$ into $\re^{N\times
N}$ which satisfies the following assumptions:
\begin{flushleft}
    (\textbf{A1}):
\end{flushleft}
\begin{description}
  \item[(i)]  $\sigma\in C^2((0,T)\times\re^N)$ and is uniformly Lipschitz in $x$ i.e. there exists a constant C such that:
  $\forall t \in [0, T]$, $x, x^{\prime} \in \re^N$,
$$\absolute{ \sigma(t,x)-\sigma(t, x^{\prime})}| \leq C \absolute{x-
x^{\prime}|}.$$
  \item[(ii)] $\sigma$ is bounded, invertible and its inverse is bounded.
  \end{description}
Note that condition (ii) is equivalent to the existence of a constant $\alpha>0$ such that for any $(t,x)$,
 \begin{equation}\label{uniellip}{\alpha} ^{-1}I \leq \sigma(t,x)\sigma^{\top}(t,x)
 \leq \alpha I\end{equation}
   i.e., $\sigma\sigma^\top$ is uniformly elliptic ($\sigma^{\top}$ is the transpose of $\sigma$).

\nd Next let $(X_t)_{t\leq T}$ be the process solution of the following stochastic
differential equation
\begin{equation}\begin{array}{c}
X_t= x+ \int_0^t \sigma(s, X_s)dB_s, \, t\leq T \,  \text{and} \
x\in \re^N.\label{original equation of x}\end{array}
\end{equation}
Since $\sigma$ verifies (A1), the process  $(X_t)_{t\leq T}$ exists and is unique (see e.g. \cite{karatzasshreve,ry} for more details).
Next let us denote by $U_1$ and $U_2$ two compact metric spaces,
meanwhile, $M_1$ and $M_2$ are the sets of the $\mathcal{P}$-measurable
processes with values in $U_1$ and $U_2$ subsets of $\re^{k_i}$, $k_i\in {\na\setminus\{0\}}$, $i=1,2$, respectively. The set $M= M_1 \times M_2$ is
called of admissible controls for players $\pi_1$ and $\pi_2$.
\mn
Now let $f$ (resp. $h_i$, $i= 1, 2$) be borelian functions from
$[0,T]\times \re^N \times U_1 \times U_2$ into $\re^N$ (resp. $\re$) and $g^i$ another borelian function from $\re^N$ to $\re$ such that for some constants $C\geq 0$ and $\gamma\geq 0$ it holds: for $i=1,2$,
\mn
${\bf (A2):}$
\begin{description}
  \item[(i)] $\absolute{f(t,x,u_1,u_2)}|\leq C(1+\absolute{x}|),$ for any $(t,x,u_1,u_2)\in [0,T]\times \re^N\times U_1\times U_2$;
   \item[(ii)] $\absolute{g^i(x)}|+\absolute{h_i(t,x,u_1,u_2)}|\leq C(1+\absolute{x}|^{\gamma})$, $\gamma \geq 1$, for any $x\in\re^N$.
\end{description}

For $(u_1,u_2)\in M$, let $P^{(u_1,u_2)}$be the probability on $(\Omega,
\mathcal{F})$ defined as follows:
\begin{equation}
dP^{(u_1,u_2)}=
\zeta(\int_0^{.}\sigma^{-1}(s,X_s)f(s,X_s,u_{1s},u_{2s})dB_s).dP\label{new
probability puv}
\end{equation}
where for any $(\mathcal{F}_t, P)$-continuous local martingale $M=
(M_t)_{t\leq T}$, the density function $\zeta(M)$ is defined by:
\begin{equation}
\zeta(M)=(\zeta(M)_t)_{t\leq T}\triangleq (\mbox{exp}\{M_t-\frac{1}{2}\langle M\rangle_t\})_{t\leq
T}\label{density function}
\end{equation}
with $(\langle M\rangle_t)_{t\leq T}$  is the increasing adapted process associated with $M$, i.e. $(M^2_t-\langle M\rangle_t)_{t\leq T}$ is a local martingale. 

Under assumptions (A1) and (A2), the non-negative measure $P^{(u_1,u_2)}$ is a probability which is equivalent to
$P$ (\cite{karatzasshreve}, p. 200) and by the Girsanov Theorem
\cite{girsanov1960transforming} the process $B^{(u_1,u_2)}= (B_t-\int_0^t
\sigma^{-1}(s,X_s)f(s, X_s, u_{1s}, u_{2s})ds)_{t\leq T}$ is an
$(\mathcal{F}_t, P^{(u_1,u_2)})$-Brownian motion and $X$ is a weak
solution of the following stochastic differential equation
\begin{equation}
dX_t=f(t, X_t, u_{1t}, u_{2t})dt+ \sigma(t, X_t)dB_t^{(u_1,u_2)},\ t\leq
T \ \mbox{and} \ X_0= x.\label{new equation of x}
\end{equation}
For i= 1,2, we define the conditional payoffs of players $\pi_1$ and $\pi_2$ respectively by\begin{equation}\begin{array}{c}
J_t^i(u_1,u_2)= E^{(u_1,u_2)}[\int_t^T h_i(s, X_s, u_{1s}, u_{2s})ds+
g^i(X_T)|\cf_t]\label{cost funciton}\end{array}
\end{equation}
where $E^{(u_1, u_2)}$is the expectation under the probability
$P^{(u_1,u_2)}$. Note that when $t=0$, $J_0^i(u_1,u_2)$ is nothing but $E^{(u_1,u_2)}[\int_0^T h_i(s, X_s, u_{1s}, u_{2s})ds+
g^i(X_T)]$ since $\cf_0$ contains only $P$-null sets and $P^{u_1,u_2}$ is equivalent to $P$.
\ms

\noindent The problem is to find a Nash equilibrium point for the game, i.e. an admissible control $(u^*, v^*)$ such that for any $(u_1, u_2)\in M$
$$J_0^1(u_1^*, u_2^*)\geq J_0^1(u_1, u_2^*) \mbox{ and } J_0^2(u_1^*, u_2^*)\geq J_0^2(u_1^*, u_2)$$ and, as much as possible, to highlight the properties of  
$(J_t^1(u_1^*, u_2^*), J_t^2(u_1^*, u_2^*))_{t\leq T}$. \qed 
\section{The PDE study of the HJBI associated with the NZSDG}
Firstly, recall once for all that we assume that the GIC introduced in Assumption (A0) is fulfilled. We will consider
the HJBI system associated with the NZSDG under different assumptions on the data.
We illustrate three cases to show the different techniques of the proof when we have lack either of continuity or boundedness.
In this way we obtain a generalization of results obtained in bounded domains
(\cite{BeFr0, PM, PM2}).
In particular we will obtain the existence of a solution of the parabolic system (suitably defined in dependence on the assumptions) in three cases:\\\\
\noindent \underline{\bf Case 1}: The data $f$, $h_i$, $g_i$ are globally bounded  and continuous with respect to all their entries and $\ol u^i(t,x,p_1,p_2)$, $i=1,2$, (see \eqref{gicintro} for the definition) are continuous.
\\

\no Example: Let us assume that $N=1$, $U_1=[0,1]$, $U_2=[-1,1]$, $f(t,x,u_1,u_2)=f_1(t,x)-u_1-u_2$, $h_1(t,x,u_1,u_2)=\bar h_1(t,x)-u_1^2$ and finally $h_2(t,x,u_1,u_2)=\bar h_2(t,x)-2u_2^2$. Then the Generalized Isaacs condition is satisfied with
$\ol u_1(t,x,p_1)=( (-\frac{p_1}{2})\wedge 1))\vee 0$ and $\ol u_2(t,x,p_2)= ((-\frac{p_2}{4})\wedge 1)\vee (-1)$, and obviously $\ol u_i$, $i=1,2$, are continuous.
\\

\noindent \underline{\bf Case 2}:
The data $f$, $h_i$, $g_i$ are globally bounded and continuous with respect to all their entries, the drift $f$ has a separate structure,  and
the feedbacks $\overline u_i(t,x,p_1,p_2), i=1,2,$ are not continuous with respect to $(p_1,p_2)$.\\

\no Example: Let us take $N=1$, $U_1=[0,1]$, $U_2=[-1,1]$, $f(t,x,u_1,u_2)=f_1(t,x)+u_1+u_2$, $h_1(t,x,u_1,u_2)=h_2(t,x,u_1,u_2)=0$, with $f_i(t,x)$ bounded and continuous. Then the Generalized Isaacs condition is satisfied with
$\ol u_1(t,x,p_1)=1_{\{p_1\geq 0\}}$ and $ \ol u_2(t,x,p_2)= 1_{\{p_2\geq 0\}}-1_{\{p_2<0\}}$ and obviously $\ol u_i$, $i=1,2$, are discontinuous.
\ms

\noindent \underline{\bf Case 3}:
The data $f$, $h_i$, $g_i$ are continuous with respect to all their entries but have a linear
growth w.r.t. $x$, the drift $f$ has a separate structure and
the feedbacks $\overline u_i(t,x,p_1,p_2)$ are not continuous with respect to $(p_1,p_2)$.
\ms

\no Note that the example of Case 2 fits also for Case 3 if we choose $f_1(t,x)$ and $g_i(x)$, $i=1,2$, continuous and with a linear growth w.r.t. $x$.

\subsection{The HJBI system associated with the NZSDG}
We denote by $R_T:=(0,T)\times \re^N$ the layer in $\re^{N+1}$.
Let us consider the following system of PDEs which stands, after inverting time, for the HJBI system of the nonzero-sum differential game introduced previously:
\begin{align}
&\frac{\partial V_{i}(t,x)}{\partial t}-
\sum _{h,k=1}^{N}a_{hk}(t,x)\frac{\partial ^2V_{i}(t,x)}{\partial x_{h}\partial x_{k}}=
H_{i}(t,x,\nabla_{x}V_{i}(t,x), \ol u_{i}(t,x), \ol
u_{j}(t,x)), \ i,j=1,2,\, i\neq j,\ \text{in } R_T,\label{sistemagenerale1}\\
&V_{i}(0,x)=g_{i}(x), \, i=1,2\quad  x\in\re^N;\label{ricontorno}\\
&\ol u_{1}(t,x)\in \argmax_{\{u_1\in U_{1}\}}H_{1}(t,x,\nabla_{x}V_{1}(t,x),
u_1, \ol
u_{2}(t,x)) \label{sistemagenerale2};\\
&\ol u_{2}(t,x)\in \argmax_{\{u_2\in U_{2}\}}H_{2}(t,x,\nabla_{x}V_{2}(t,x),\ol
u_{1}(t,x),u_2),\label{sistemagenerale3}
\end{align}
where
$a=\frac{1}{2}\sigma \sigma^{\top}$
is the matrix with entries $a_{hk}$, $h, k=1,\dots N$ and the Hamiltonian $H_i$ are
defined in \eqref{Hamdef}.\\\\
Recall that, from assumption $\textbf{(A1)}$, the matrix $a(t,x) \in
C^{2}(R_{T})$, is bounded and uniformly elliptic, in the sense that for
all $(t,x)\in R_{T}$ and for all $\xi\in {\re}^{N}$,
\begin{align}
 \alpha^{-1} |\xi| ^{2}\leq \sum_{h,k=1}^{N}a_{hk}(t,x)\xi_{h}\xi_{k}\leq
\alpha |\xi| ^{2}\,\,(\a>0).\label{ipotesia2}
\end{align}
Let $\Omega\subset \re^N$ be a bounded open domain and let us define
 $\Omega_{T}:\equiv (0,T)\times \Omega$ and $\partial_{p}\Omega_{T}:\equiv \big((0,T)\times \partial\Omega\big)\cup
\big(\{t=0\}\times\Omega\big)$. We denote by $H^{1+\alpha}(\ol\Omega_{T})$, $\alpha\in (0,1)$, the set of functions $v(t,x)$ such that
$v$ is a $\a$-H\"{o}lder continuous function in $\overline\Omega_T$ together with its spatial derivatives $\frac{\partial v}{\partial x_i}$, $i=1,...,N$. The norm in $H^{1+\alpha}$ is denoted by $|v|^{(1+\alpha)}$.
We denote by
$W_q^{1,2}(\Omega_{T})$, $q>1$, the set of functions $v(t,x)$ such that
$v$ and its weak derivatives  $\frac{\partial v}{\partial t}$,
$\frac{\partial v}{\partial x_i}$, $\frac{\partial^2 v}{\partial x_i\partial x_j}$ belong to $L^q(\Omega_{T})$.
The norm in $W_q^{1,2}(\Omega_T)$ is denoted by $\|v\|_q^{(2)}$.\\
\subsection{Bounded continuous data and feedbacks}\label{Contf}
Let us study now HJBI system
\eqref{sistemagenerale1}-\eqref{sistemagenerale3} under the assumptions of Case 1, i.e.,  the functions $f$, $h_i$, $g_i$, $i=1,2$, are globally bounded  and continuous with respect to all their entries. Precisely we assume that:

\ms
\noindent {\bf Assumption (H1)}:
\begin{align}
&\mbox{(i) }\text{ The functions } f(t, x, u_{1}, u_{2}) \text{ and } h_i(t, x, u_{1}, u_{2}), i=1,2, \text{ are globally bounded in } R_T\nonumber\\
&  \text{ and continuous in }R_T\times U_1\times U_2 ;\label{fi}\\
&\mbox{(ii) } \text{ For }i=1,2,\,\,g_{i}(x) \in H^{1+\alpha}(Q),
\alpha\in (0,1),
\text{ for any bounded } Q\subset \re^N,\nn\\
&\text{ and it is bounded and continuous in }\re^N; \label{ipotesig}\\
&\mbox{(iii) } \text{ For } i=1,2,\,\,\text{the mapping}\nonumber\\
&(p_1,p_2)\in \re^{N+N}\mapsto H_i(t, x, p_i, \bar u_1(t,x,p_1,p_2), \bar u_2(t,x,p_1,p_2))\in \re \text{ is continuous.}\label{giccont}
\end{align}
System (\ref{sistemagenerale1})-(\ref{ricontorno}) is a
Cauchy problem for a quasilinear  uniformly parabolic system in the layer
$R_T$ with equations strongly coupled by the functions\\
$H_{i}(t, x,\nabla_{x}V_i(t,x), (\ol u_1,\ol u_2)(t,x,\nabla_{x}V_{1}(t,x), \nabla_{x}V_{2}(t,x)))$, $i=1,2$.

\begin{definition}\label{def}
$(V_1,V_2)$ is a {\em strong} solution
of the system (\ref{sistemagenerale1})-(\ref{sistemagenerale3}),  if
\begin{align}
&a)\ V_{1}(t,x),\ V_{2}(t,x) \in L_{\infty}(R_T),\\
&b)\ V_{1}(t,x),\ V_{2}(t,x) \in H^{1+\alpha}(\ol \Omega_{T})\cap
W^{1,2}_{q}(\Omega_{T}), \\& \qq \qq \qq \text{ where for any bounded subdomain } \Omega \subset \re^N,\ \Omega_{T}=(0,T)\times\Omega, \ \alpha\in (0,1),\ q>N+2;\nn\\
&c)\text{ Equations}\
(\ref{sistemagenerale1}),(\ref{sistemagenerale2}),(\ref{sistemagenerale2})
\text{ hold almost everywhere in } \Omega_T \text{ and }\
(\ref{ricontorno}) \text{ holds in } \Omega.\nn
\end{align}
\end{definition}

\begin{theorem}\label{existence}
Under assumptions \eqref{ipotesia2} and (H1), there exists a strong solution
$(V_1,V_2)$  of
the parabolic system (\ref{sistemagenerale1})-(\ref{sistemagenerale3})
in the layer $R_T$.

\end{theorem}
\begin{proof} To prove the existence of a strong solution in any $\Omega_T=(0,T)\times \Omega$, where
$\Omega\subset \re^N$ is any bounded domain, let us consider the
following problem in a sequence of expanding domains of the form
$B_{R,T}:=(0,T)\times B(0,R)$ where $B(0,R):=\{|x|<R\}$ (clearly, if
$R\rightarrow +\infty$, $B_{R,T}\rightarrow R_T$):

\begin{align}
&\frac{\partial V_{1}^R(t,x)}{\partial t}-
\sum _{h,k=1}^{N}a_{hk}(t,x)\frac{\partial ^2V_{1}^R(t,x)}{\partial x_{h}\partial x_{k}}\nn\\
&\qquad\qquad = H_{1}(t, x,\nabla_{x}V_{1}^R(t,x), \ol u_{1}(t,x,\nabla_{x}V_{1}^R, \nabla_{x}V_{2}^R), \ol
u_{2}(t, x,\nabla_{x}V_{1}^R, \nabla_{x}V_{2}^R)), \,\mbox{in}\ B_{R,T}\label{sistemageneraleD};\\
&\frac{\partial V_{2}^R(t,x)}{\partial t}-
\sum _{h,k=1}^{N}a_{hk}(t,x)\frac{\partial ^2V_{2}^R(t,x)}{\partial x_{h}\partial x_{k}}
\nn\\
&\qquad \qquad =H_{2}(t,x,\nabla_{x}V_{2}^R(t,x), \ol u_{1}(t,x, \nabla_{x}V_{1}^R, \nabla_{x}V_{2}^R), \ol
u_{2}(t,x, \nabla_{x}V_{1}^R, \nabla_{x}V_{2}^R)), \, \mbox{in}\ B_{R,T};\\
&V_{i}^R(t,x)=g_{i}(x), \quad i=1,2,\quad
\mbox{in}\ \partial_pB_{R,T}:=\{(t,x)\in R_T, |x|=R\}\cup \{(t,x)\in R_T, t=0\}.\label{contornoD}
\end{align}

Remark that this auxiliary problem is compatible with our initial nonzero-sum game problem.
In fact if a solution $(V_1^R(t,x),V_2^R(t,x))$ exists and both functions belong to $W^{1,2}_p(B_{R,T})$ then, by  setting $(\ol V_1^R(t,x),\ol V_2^R(t,x))=(V_1^R(T-t,x),V_2^R(T-t,x))$ and using the It\^o-Krylov formula (see e.g. \cite{krylov}, Theorem 2.10.1) we have the following characterization: for $i=1,2$,
\begin{align}\label{eq:vlocal}
&\ol V_i^R(0,x)=J^R_{i}(\ol u_{1}, \ol u_{2}):=
E^{\bar u_1,\bar u_2}\bigg\{\int^{\tau_R}_{0} h_{i}(s,X_s,(\ol u_{1},\ol u_{2})(s))ds
+g_{i}(X_{\tau_R})\bigg\}
\end{align}
where:

(i) $(X_s)_{s\leq T}$ is the stochastic process defined in (\ref{original equation of x}) ;

(ii) $\tau_R\equiv T \wedge \inf\{s\geq t, X(s)\notin B_R \}$ ;

(iii) $(\ol u_{1},\ol u_{2}):=((\ol u_{1},\ol u_{2})(s))_{s\leq T}=((\ol u_{1},\ol u_{2})(s,X_s,
\nabla_{x}\ol V_{1}^R(s,X_s),\nabla_{x}\ol V_{2}^R(s,X_s)))_{s\leq T}$.
\mn
Moreover the pair $(\ol u_1,\ol u_2)$ is a Nash equilibrium point for the nonzero-sum differential game defined with the same data $f$, $h_i$, $g_i$, $U_i$, i=1,2, etc. but which terminates at the random time $\tau_R$ (see Theorem \ref{nashpoint} for more details). Finally note that if $x\in\partial B(0,R)$  then $\tau_R=0$ and $V_i^R(0,x)=g_i(x)$.

To prove the existence of a solution $(V_1^R, V_2^R)$ of
problem (\ref{sistemageneraleD})-(\ref{contornoD})
we use a standard bootstrap argument and we find uniform estimates which will allow us to prove the convergence to the solution we are looking for.

Let $\{V_{1n}^R(t,x),V_{2n}^R(t,x)\}$, $n\geq 1$, be the solution of the following system:
\begin{align}
&\frac{\partial V_{1n}^R}{\partial t}-
\sum _{h,k=1}^{N}a_{hk}(t,x)\frac{\partial ^2 V_{1n}^R}{\partial x_{h}\partial x_{k}}\nn\\& \qquad\qquad
=H_{1}(t,x,\nabla_{x}V_{1n}^R(t,x),
(\ol u_{1},\ol
u_{2})(t,x, \nabla_{x}V_{1(n-1)}^R(t,x), \nabla_{x}V_{2(n-1)}^R(t,x))
)\mbox{ in }\ B_{R,T};\label{sistemageneraleDn}\\
&\frac{\partial V_{2n}^R}{\partial t}-
\sum _{h,k=1}^{N}a_{hk}(t,x)\frac{\partial ^2 V_{2n}^R}{\partial x_{h}\partial x_{k}}\nn
\\& \qquad\qquad
=H_{2}(t,x,\nabla_{x}V_{2n}^R(t,x),
(\ol u_{1},\ol
u_{2})(t,x, \nabla_{x}V_{1(n-1)}^R(t,x), \nabla_{x}V_{2(n-1)}^R(t,x))
) \mbox{ in } B_{R,T} ;\\
&V_{in}^R(t,x)=g_{i}(x),\ i=1,2, \mbox{ on }\ \partial_{p}B_{R,T}.\label{contornon}
\end{align}

Note that this is a linear system of parabolic equations, but those latter are decoupled. Next as $f$, $h_i$ and $g_i$, $i=1,2$, are bounded functions in $R_T$, then from Theorem 9.1 p.341 of \cite{LSU} and Lemma 3.3, Chapter 2 p.80 \cite{LSU}
there exists an unique solution of problem (\ref{sistemageneraleDn})-(\ref{contornon}), $V_{1n}^R, V_{2n}^R\in W^{1,2}(B_{R,T})$
such that
\begin{equation}\label{Vnunif}
\|V_{in}^R\|^{(2)}_{q,B_{R,T}}\leq  C\bigg(\|f\|_{q,B_{R,T}},\|h_i\|_{q,B_{R,T}}, \|g_i\|^{2-1/q}_{q,\partial_pB_{R,T}}\bigg),\ i=1,2,
\end{equation}
where $C$ is a constant which does not depend on  $n$ and $R$.
By means of the Sobolev embedding theorem we also have
\begin{equation}\label{Vnhol}
\|V_{in}^R\|^{(1+\alpha)}_{B_{R,T}}\leq  C,\ \mbox{
$\alpha=1-\frac{N+2}{q}$},\ i=1,2,
\end{equation}
where $C$ is a constant which does not depend on $n$ and on $R$.

By \eqref{Vnhol} and Ascoli-Arzel\`a Theorem, we can extract two subsequences, which we denote again by
$V_{1n}^R$, $V_{2n}^R$ such that
\begin{eqnarray}\label{convergenzeC0aff}
&&V_{in}^R\rightarrow V_{i}^R,\
\displaystyle\frac{\partial V_{in}^R}{\partial x_{h}}\rightarrow
\displaystyle \frac{\partial V_{i}^R}{\partial x_{h}},\ \mbox{in}\ C^{0}(B_{R,T}),\ i=1,2,\
h=1,\ldots N,
\end{eqnarray}
and, from \eqref{Vnunif} and the weak precompactness of the unit ball of $W^{2,1}_q$, we have also
\begin{eqnarray}\label{convergenzedeboliaff}
\!\!\!\!&&\displaystyle\frac{\partial V_{in}^R}{\partial t}\rightharpoonup
\displaystyle\frac{\partial V_{i}^R}{\partial t},\
\displaystyle\frac{\partial ^2V_{in}^R}{\partial x_{h}\partial x_{k}}\rightharpoonup
\displaystyle\frac{\partial ^2V_{i}^R}{\partial x_{h}\partial x_{k}},
\ \mbox{weakly in } L_{2}(B_{R,T}),\ i=1,2,\ h,k=1,\ldots N .
\end{eqnarray}
From (\ref{convergenzeC0aff}), (\ref{convergenzedeboliaff}),
$
V_{1}^R,\ V_{2}^R\in H^{1+\alpha}(B_{R,T})\cap
W^{1,2}_{q}(B_{R,T})$, with
$\alpha=1-\frac{N+2}{q}$.
\ms

Next the following decomposition holds true:
$$\begin{array}{l}
H_{1}(t,x,\nabla_{x}V_{1n}^R(t,x),
(\ol u_{1},\ol
u_{2})(t,x, \nabla_{x}V_{1(n-1)}^R(t,x), \nabla_{x}V_{2(n-1)}^R(t,x))
)\\\\=(\nabla_{x}V_{1n}^R(t,x)-\nabla_{x}V_{1(n-1)}^R(t,x))f(t,x,
(\ol u_{1},\ol
u_{2})(t,x, \nabla_{x}V_{1(n-1)}^R(t,x), \nabla_{x}V_{2(n-1)}^R(t,x))
)+ \\\\\qquad\qquad
H_{1}(t,x,\nabla_{x}V_{1(n-1)}^R(t,x),(\ol u_{1},\ol
u_{2})(t,x, \nabla_{x}V_{1(n-1)}^R(t,x), \nabla_{x}V_{2(n-1)}^R(t,x)
)).
\end{array}$$
The first term, as $n\rw \infty$, converges to $0$ since $f$ is bounded and \\$(\nabla_{x}V_{1n}^R(t,x)-\nabla_{x}V_{1,n-1}^R(t,x)) \rw _n0$ while the second one converges to
$$
H_{1}(t,x,\nabla_{x}V_{1}^R(t,x),(\ol u_{1},\ol
u_{2})(t,x, \nabla_{x}V_{1}^R(t,x), \nabla_{x}V_{2}^R(t,x)
))
$$
by the continuity of assumption (A3)-(ii). We can do the same for the quantity
$$H_{2}(t,x,\nabla_{x}V_{2n}^R(t,x),
(\ol u_{1},\ol
u_{2})(t,x, \nabla_{x}V_{1(n-1)}^R(t,x), \nabla_{x}V_{2(n-1)}^R(t,x))
)
$$
which converges, as $n\rw \infty$, to
$$
H_{2}(t,x,\nabla_{x}V_{2}^R(t,x),(\ol u_{1},\ol
u_{2})(t,x, \nabla_{x}V_{1}^R(t,x), \nabla_{x}V_{2}^R(t,x)
)).
$$
Going back now to (\ref{sistemageneraleDn})-(\ref{contornon}), take the limit w.r.t $n$ to obtain that $\{V_{1}^R,V_{2}^R\}$ solve equations
(\ref{sistemageneraleD})-(\ref{contornoD}) almost everywhere in $B_{R,T}$ and
$V_{i}^R=g_{i}$, $i=1,2$, on $\partial_{p}B_{R,T}$.

Moreover from the boundedness of the data of the problem, applying the maximum principle (\cite{LSU}, Theorem 2.1, p.13) in
$B_{R,T}$
we obtain that the solution $V_i^R$ of (\ref{sistemageneraleD})-(\ref{contornoD}) is such that
\begin{equation}\label{Linfinito}
\|V_1^R, V_2^R\|_{\infty}\leq C,
\end{equation}
where $C$ does not depend on $R$ hence they are uniformly bounded.
From the previous estimate \eqref{Linfinito} we can say that for any
$R_0>0$ and $V_{i}^R$ with $R>R_0$ we have
\begin{equation}\label{VR0}
\|V_{i}^R\|^{(2)}_{q,B_{R_0,T}}\leq  C(R_0),\ i=1,2.
\end{equation}
where the constant $C(R_0)$ depends on $R_0$ but not on $R$.

Now by employing the usual diagonal process we can extract from the sequence $\{V_i^R\}$ a subsequence which we call again $\{V_i^R\}$ that converges together with the first derivatives
 $\nabla_xV_{i}^{R}$ at each point of $R_T$ to some functions
 $V_i$, and  such that $\nabla_tV_{i}^{R}$, $D^2_{xx}V_{i}^{R}$  converge weakly in $L_2(\Omega_T)$ to  $\nabla_tV_{i}$, $D^2_{xx}V_{i}$ respectively for any $\Omega_T\subset R_T$ with $\Omega$ bounded subset of $\re^N$. Now from (\ref{Linfinito}) $V_i(t,x)$ are bounded in $R_T$.
 Hence
 \begin{equation}\label{V21}
\|V_{i}\|^{(2)}_{q, \Omega_T}\leq  C(\ol{\Omega}_T),\ i=1,2,
\end{equation}
 for any $\Omega_T\subset R_T$.
 Moreover $V_i(t,x)$, $i=1,2$, solve problem (\ref{sistemagenerale1})-(\ref{ricontorno})
 in any $\Omega_T\subset R_T$ with $\Omega$ bounded subset of $\re^N$, i.e.
 is a strong solution of the problem (see also Section 8, p.492-493 of \cite{LSU}).
\end{proof}

\subsection{Bounded data and discontinuous Hamiltonian.}\label{discontf}
In this subsection we consider the case where the generalized Isaacs
condition (A0) is satisfied with discontinous functions $\ol
u^i(t,x,p_1,p_2)$, $i=1,2$, w.r.t $(p_1,p_2)$. To better understand
the problem we start with
 an example where the feedback can be written in an explicit way.
This problem was considered in \cite{PM} for bounded domains and in
\cite{HaMu} in $\re^N$.\\ 
For the sake of brevity here in the following HJBI equation and also in the next section we denote by $V_y$ the derivative $\frac{\partial V}{\partial y}$ with respect to a generic variable $y$.\\
We take an affine structure of $f$ and
$h_i$, i.e.,
\begin{align}
&f(t, x,u_{1}, u_{2})= f_{1}(t, x) u_{1}+f_{2}(t, x) u_{2}, \text{ where for }i=1,2,\nn\\
&\qq f_{i}:\ (0,T)\times\re^N \rightarrow \re,\ f_{i}\in
C^{1}([0,T]\times \re^N)\ \text { and bounded } ;\label{faffine}\\\nn\\
&h_i:\ (0,T)\times\re^N\times U_{1}\times
U_{2}\rightarrow  {\re},\,h_{i}(t, x,u_{1}, u_{2})= h_{i}(t, x) u_{i} \mbox{ with }\nn\\
&\qq h_{i}:\ (0,T)\times\re^N\rightarrow \re,\ h_{i}\in
C^{1}([0,T]\times \re^N)\ \text { and bounded }, \,i=1,2.\label{payoffaffine}
\end{align}
From (\ref{faffine})-(\ref{payoffaffine}) we have:
\begin{align}\label{prehamlineare}
&H_{1}(t, x,p,u_{1}, u_{2})= (p \cdot f_{1}(t, x) + h_{1}(t, x)) u_1+
p \cdot f_{2}(t, x)u_{2},\\
&H_{2}(t, x,p,u_{1},u_{2})= (p \cdot f_{2}(t, x) + h_{2}(t, x)) u_2+
p \cdot f_{1}(t, x)u_{1}.\nn
\end{align}
We take
as control sets
\begin{equation}\label{set}
 U_1=U_2=[0,1].
\end{equation}
In this case it is possible to
find an explicit expression to
 $\argmax_{\{u_{i}\in U_{i}\}}H_{i}(t, x,p,u_{i})=\Heav(p\cdot f_i(t, x)+h_i(t, x))$, $i=1,2$.
Here $\Heav(\eta)$ is the set valued Heaviside
graph, $\Heav(\eta)=1$, if $\eta>0$, $\Heav(\eta)=0$ if $\eta<0$, $\Heav(0)=[0,1]$, i.e. $\Heav(\eta)$
is a multivalued function from $\re$ to $\re$ that associates to each point $\eta\in\re$ a set $\Heav(\eta)\subseteq \re$.
In this case we can explicitely see  that the optimal feedbacks
$\ol u_{i}(t,x,p)\in \argmax_{\{u_{i}\in U_{i}\}}H_{i}(t, x,p,u_{i})=\Heav(p\cdot f_i(t,x)+h_i(t, x))$
are not continuous with respect to the variable $p$ and
the generalised Isaacs condition (A0) is satisfied with discontinuous functions $\ol u^i$, $i=1,2$, w.r.t $(p_1,p_2)$.\\
Hence the terms on the right hand sides
contain multivalued functions and the system
(\ref{sistemagenerale1})-(\ref{ricontorno}) becomes
\begin{align}
&V_{1t}-\sum_{h,k}a_{hk}V_{1x_{h} x_{k}}
= (\nabla_xV_{1}\cdot
f_{1}+
h_{1}) \ol u_1+\nabla_{x}V_{1}\cdot f_{2}\, \ol u_2,\label{HH}\\
&V_{2t}-\sum_{h,k}a_{hk}V_{2x_{h} x_{k}}
=(\nabla_xV_{2}\cdot 
f_{2}+
h_{2}) \ol u_2+\nabla_{x}V_{2}\cdot f_{1}\, \ol u_1,\nn\\
&\ol u_{1}(t,x, \nabla_{x}V_{1})\in \mathrm{Heav}((\nabla_{x} V_{1}\cdot f_{1}
+h_{1})(t,x)),\nn\\
&\ol u_{2}(t,x, \nabla_{x}V_{2})\in \mathrm{Heav}((\nabla_{x} V_{2}\cdot f_{2}
+h_{2})(t,x)),\nn\\
&V_i(0,x)=g_i(x)\nn.
\end{align}
that can be written also as: $\forall (t,x)\in R_T$, 
\begin{align}
&V_{1t}-a_{hk}V_{1x_{h} x_{k}} \in (V_{1x}\cdot f_{1}+
h_{1}) \mathrm{Heav}(V_{1x}\cdot f_{1}
+h_{1})+V_{1x}\cdot f_{2}\, \mathrm{Heav}(V_{2x}\cdot
f_{2}
+h_{2})\,;\nn\\
&V_{2t}-a_{hk}V_{2x_{h} x_{k}} \in (V_{2x}\cdot f_{2}+
h_{2}) \mathrm{Heav}(V_{2x}\cdot f_{2}
+h_{2})+V_{2x}\cdot f_{1}\, \mathrm{Heav}(V_{1x}\cdot
f_{1}
+h_{1})\,;\nn\\
&V_i(0,x)=g_i(x),\ x\in\,\re^N.
\end{align}
In the following theorem we strongly use the explicit expression of the Hamiltonian
to get the existence result.
\begin{theorem}\label{existence3}
Under assumptions \eqref{ipotesia2},
(\ref{faffine})-(\ref{payoffaffine}), (\ref{set}), (\ref{ipotesig})
there exists a strong solution $(V_1,V_2)$  of the parabolic
system (\ref{HH}) in the layer $R_T$.
\end{theorem}
\begin{proof}
We follow the procedure used in \cite{PM} in bounded domains.
We approximate the Heaviside graph by a smooth sequence $H_n$ and we consider the solution $V_{in}^{R}$ of the corresponding
 Dirichlet problem in a bounded domain $B_{R,T}$ as in the previous theorems
 where the boundary condition is
 $V_{in}^{R}=g_i(x)$ on $\partial_pB_{R,T}$.
 Note that, as in the previous section (see proof of Theorem \ref{existence}), this auxiliary Dirichlet problem is compatible with our game problem in all $\re^N$.\\
 There exists an unique solution
 $\{V_{1n}^R, V_{2n}^R\}\in W^{1,2}(B_{R,T})$ and
 from the  boundedness of $H_n$ uniform on $n$ and the
 boundedness of $g_i$, we obtain uniform estimates on $n$:
 \begin{equation}\label{VnunifH}
\|V_{in}^R\|^{(2)}_{q,B_{R,T}}\leq  C\bigg(\|H_n\|_{q,B_{R,T}}, \|g_i\|^{2-1/q}_{q,\partial_pB_{R,T}}\bigg)=C(B_{R,T}),\ i=1,2.
\end{equation}
where the constant C is independent on $n$.\\
By means of the Sobolev embedding theorem we also have
\begin{equation}\label{VnholH}
\|V_{in}^R\|^{(1+\alpha)}_{B_{R,T}}\leq  C(B_{R,T}),\
\alpha=1-\mbox{$\frac{N+2}{q}$},\ i=1,2,
\end{equation}
where $C$ is independent of $n$.\\
Still following the procedure of \cite{PM} we find a solution of Problem
\eqref{HH}, $\{V_1^R, V_2^R\}$ in $B_{R,T}$ and, as in \eqref{VR0},
 we can say that for any $R_0>0$ and $V_{i}^R$ with $R>R_0$ we have
\begin{equation}\label{VRH}
\|V_{i}^R\|^{(2)}_{q,B_{R_0,T}}\leq  C(R_0),\ i=1,2,
\end{equation}
where the constant $C(R_0)$ depends on $R_0$ but not on $R$.
Hence passing to the limit as $R\rightarrow +\infty$ we obtain a strong solution
of  (\ref{HH}).
\end{proof}

After the previous example
we consider more general cases. Actually assume that: either
\ms

\noindent i) The game has separate
dynamics and running payoffs and the functions are $\re$-valued\\
or
\ms

\noindent ii) The control sets
are multidimensional compact sets and the dynamics is affine as in \eqref{faffine}.
\ms

For bounded domains this type of problem was considered in
\cite{PM2}, here we want to extend it to unbounded domains. We make
the following assumption: \medskip

\noindent {\bf Assumption (H2)}:
\begin{align}
& a) \,U_{i} \text{ are convex compact sets in } {\re^{k_i}},\ k_i\in\na,\ k_i\geq 1, \ i=1,2 ;\label{compatti}\\
& b) \,f: (0,T)\times \re^N\times U_{1}\times U_{2}\rightarrow  {\re}^{N},
f(t, x,u_{1}, u_{2})=f_{1}(t, x,u_{1})+f_{2}(t, x, u_{2})\\&\qq \mbox{with }f_{i}:
\ (0,T)\times\re^N\times U_{i}\rightarrow {\re}^{N}, f_{i}\in C^{1}([0,T]\times \re^N\times U_{i}) \mbox{ and bounded},\
i=1,2;\nn\\
&c)\,h_i:\ (0,T)\times \re^N \times U_{1}\times U_{2}\rightarrow  {\re},
\,\,h_{i}(t, x,u_{1}, u_{2})= h_{i}(t, x, u_{i}),\mbox{ with }\nn\\
& \qq h_{i}:\ (0,T)\times\re^N \times U_{i} \rightarrow \re, h_{i}\in C^{1}([0,T]\times\re^N\times U_{i}) \mbox{
and bounded},\,
i=1,2;\label{payoffseparato} \\
&d)\,\,\forall p\in\re^N, A_i(t,x,p):= \argmax_{\{u_{i}\in
U_{i}\}}\big(p^\top\cdot
f_{i}(t,x,u_i)+h_{i}(t,x,u_i)\big) \nn \\&\qquad
 \mbox{ are convex sets in } \re^{k_i}, i=1,2. \label{assumeconvex}
\end{align}

We state now the existence theorem in the case i).
\begin{theorem}\label{existence4}
Let us suppose that $N=1$. Under assumptions \eqref{ipotesia2},
\eqref{ipotesig}, and (H2),
i.e.(\ref{compatti})-\eqref{assumeconvex}, there exists a strong
solution ($V_{1},\ V_{2}$)  of the parabolic system (\ref{HH}) in
the layer $R_T$.
\end{theorem}
\begin{proof}
 We follow the lines of the proof in \cite{PM2}.
 By Cellina's approximation theorem we find a sequence $A_{in}(t,x,p)$ in a $1/n$-neighbourhood of
 the graph of \\
 $A_i(t,x,p):= \argmax_{\{u_{i}\in U_{i}\}}\big(p\cdot f_{i}(t,x,u_i)+h_{i}(t,x,u_i)\big)$.
 We consider the solution $V_{in}^{R}$ of the corresponding
 Dirichlet problem in a bounded domain $B_{R,T}$ as in Theorem \eqref{existence3}.
 We obtain uniform estimate on $n$
 and the convergence to $V_{i}^{R}$ solution of system (\ref{HH}) in
 $B_{R,T}$.
 Hence by the diagonal procedure and passing to the limit as $R\rightarrow +\infty$ we obtain a strong solution
of  (\ref{HH}) in the layer $R_T$.
 \end{proof}
\begin{remark}The case ii) can be treated analogousy if additionnally to assumptions (H2) we require that 
$$
f(t, x,u_{1}, u_{2})= f_{1}(t, x) u_{1}+f_{2}(t, x) u_{2}.$$
However we do not require $N=1$ as in Theorem \eqref{existence4}. \qed  
\end{remark}
\subsection{Unbounded data, discontinuous Hamiltonian, unbounded domains.}\label{allunbound}
In this section we want to study from the PDEs point of view
the game studied in the paper \cite{HaMu} with probabilistic tools.
We consider a stochastic game where the drift of the dynamics of the system is of type
$$f(t, x,u_{1}, u_{2})=f_{1}(t, x)u_{1}+f_{2}(t, x) u_{2}+\varphi(t,x),$$
where
\begin{equation}\label{linear} \left\{\begin{array}{l}
a)\,\, f_i,\,i=1,2,\,\mbox{ and }\varphi \text{ are continuous };  \\\\
b) \,\,\max\{|f_1(t,x)|,|f_2(t,x)|,|\varphi(t,x)|\}\leq C(1+|x|),\,\forall  (t,
x)\in R_T.    \end{array} \right.
\end{equation}
Moreover we suppose that the terminal payoffs satisfy: 
\begin{equation}\label{gpoly} \mbox{For }i=1,2, \,\, 
g_i \text{ is continuous and }|g_i(x)|\leq C(1+|x|^{\beta}),\ \beta\geq 1,\   x\in \re^N.
\end{equation}
Without loss of generality we suppose that the running payoffs
$h_i=0$, $i=1,2$, and $U_1=U_2=[0,1]$.
In this case the Hamiltonians become

\begin{align}\label{prehamlineareinfi}
&H_{1}(t, x,p,u_{1}, u_{2})= p\cdot \left(f_{1}(t,
x) u_1+
f_{2}(t, x)u_{2}+ \varphi(t,x)\right),\\
&H_{2}(t, x,p,u_{1},u_{2})=p\cdot \left(f_{1}(t, x)
u_1+ f_{2}(t, x)u_{2}+ \varphi(t,x)\right).\nn
\end{align}
Hence also in this case the optimal feedbacks \\$\ol u_{i}(t,x,p)\in
\argmax_{\{u_{i}\in U_{i}\}}H_{i}(t,
x,p,u_{i})=\Heav(p\cdot f_i(t,x))$ are not
continuous with respect to the variable $p$, and then the Hamiltonians do not have continuous dependence on
$\nabla_xV_{1}$, $\nabla_xV_{2}$.\\
The system (\ref{sistemagenerale1})-(\ref{sistemagenerale3})
becomes: $\forall (t,x)\in R_T$, 
 \begin{align}
&V_{1t}-\sum_{h,k}a_{hk}V_{1x_{h} x_{k}} =
\nabla_xV_{1}\cdot (
f_{1}\ol u_1+f_{2}\, \ol u_2 +\varphi(t,x)),\label{HHD}\\
&V_{2t}-\sum_{h,k}a_{hk}V_{2x_{h} x_{k}} =
\nabla_xV_{2}\cdot (
f_{1}\ol u_1+f_{2}\, \ol u_2 +\varphi(t,x)),\nn\\&V_i(0,x)=g_i(x),\nn\\
&\ol u_{1}(t,x, \nabla_{x}V_{1})\in \mathrm{Heav}(\nabla_{x} V_{1}\cdot f_{1}(t,x)),\nn\\
&\ol u_{2}(t,x, \nabla_{x}V_{2})\in \mathrm{Heav}(\nabla_{x}V_{2}\cdot f_{2}(t,x)).\nn
\end{align}
and can be written also as: $\forall (t,x)\in R_T$, 
\begin{align}\label{unbound}
&V_{1t}-a_{hk}V_{1x_{h} x_{k}}
\in \nabla_xV_{1}\cdot \left(
f_{1} \mathrm{Heav}(\nabla_xV_{1}\cdot f_{1})+f_{2}\, \mathrm{Heav}(\nabla_xV_{2}\cdot f_{2})+\varphi(t,x)\right)\,;\\
&V_{2t}-a_{hk}V_{2x_{h} x_{k}} \in \nabla_xV_{2}\cdot \left(
f_{1} \mathrm{Heav}(\nabla_xV_{1}\cdot f_{1})+f_{2}\, \mathrm{Heav}(\nabla_xV_{2}\cdot f_{2})+\varphi(t,x)\right)\,;\nn\\
&V_i(0,x)=g_i(x),\ x\in\re^N.\nn
\end{align}
This is a system with discontinuous and unbounded terms in an unbounded domain.
Here below we obtain two existence results. The first one gives a weak solution in all the strip $R_T$ and could be considered as a general result for systems with unbounded coefficients and discontinuous Hamiltonians. The second one comes directly from the procedure used in the previous sections and allows us to find a more regular solution but only in the bounded subdomains of  $R_T$.
We want to write here both the results even if for the existence of Nash equilibria is
 sufficient only the second one.\\
To obtain the first existence result we give a
suitable definition of solution in the spaces $L^{2,\infty}(R_T)$
and $L^{2,2}(R_T)$, following \cite{Ar} and \cite{Ik}:
\begin{definition}
The space $L^{p,q}(Q_T):=L^{q}\left[(0,T), L^p(Q)\right]$ is the
space where we define the following norm: For $w\in L^{p,q}(Q_T)$ we
have
$$\|w\|_{p,q,Q}=\left\{\int_0^T(\int_Q |w(t,x)|^p dx) ^{q/p} dt \right\}^{1/q}.$$
In the case either $p$ or $q$ are infinite $\|w\|_{p,q,Q}$ is definite in a
similar way using $L^{\infty}$ norm:
$$\|w\|_{p,\infty,Q}=essup_{(0,T)}(\int_Q |w(t,x)|^p dx) ^{1/p}.$$
\end{definition}
\begin{definition}\label{defweak} A pair $\{V_1,V_2\}$ is said a weak
solution of Problem \eqref{unbound} in $R_T=(0,T)\times \re^N$ for
the initial condition $g_i(x)\in L^2_{loc}(\re^N)$ if $V_i(t,x)\in
L^{\infty}\left[(0,T), L^2_{loc}(\re^N)\right]\cap L^{2}[ (0,T),
H^{1,2}_{loc}(\re^N)]$ and if $V_i$, $i=1,2$, satisfy
\begin{align}
&\int\!\!\!\int_{R_T}\{-V_{1}\Phi_t+a_{hk}V_{1x_{h}}\Phi_{x_{k}} +a_{hkx_k}V_{1x_{h}}\Phi \label{weakversion}
\\
&\qquad-\nabla_xV_{1}\cdot \left( f_{1}
\mathrm{Heav}(\nabla_xV_{1}\cdot f_{1})+f_{2}\,
\mathrm{Heav}(\nabla_xV_{2}\cdot f_{2})+\varphi(t,x)\right)
\Phi(t,x)\}\, dx\,dt=0,\nn\\
&\int\!\!\!\int_{R_T}\{-V_{2}\Phi_t+a_{hk}V_{2x_{h}}\Phi_{x_{k}}+{a_{hk}}_{x_{k}}V_{2x_{h}}\Phi\nn\\
&\qquad- \nabla_xV_{2}\cdot \left( f_{2}
\mathrm{Heav}(\nabla_xV_{2}\cdot f_{2})+f_{1}\,
\mathrm{Heav}(\nabla_xV_{1}\cdot 
f_{1})+\varphi(t,x)\right)\Phi(t,x)\}\,dx\,dt=0,\nn
\end{align}
for any $\Phi\in C^1_0(R_T)$. Moreover $V_i(0,x)= g_i(x)$, $x\in\re^N$, $i=1,2$.
\end{definition}

\begin{theorem}
Under assumption \eqref{ipotesia2}, \eqref{linear}, \eqref{gpoly} there exists a weak solution of the Cauchy problem \eqref{unbound}.
Moreover the solution $\{V_1$, $V_2\}$ are locally H\"{o}lder continuous on $R_T$ and satisfy the following estimate
\begin{equation}\label{growth}
|V_i(t,x)|\leq C(1+|x|^{\beta}),\ \beta\geq 1,\ \forall (t,x)\in R_T,
\end{equation}
where $\beta$ is the growth exponent of assumption \eqref{gpoly}.
\end{theorem}
\begin{proof}
We use the results of \cite{Ar} and \cite{Ik} where linear parabolic Cauchy problems with possibly discontinuous terms respectively in bounded and unbounded domains are considered.
To get a linear system, we use a bootstrap argument by defining a sequence of solutions
$V_{1}^n$, $V_{2}^n$ of the following problem
\begin{align}
&V_{1t}^n-\sum_{h,k}a_{hk}V_{1x_{h} x_{k}}^n
= \nabla_xV_{1}^n\cdot(
f_{1}\ol u_1^{(n-1)}+f_{2}\, \ol u_2^{(n-1)} +\varphi),\ in \ R_T;\label{HHn}\\
&V_{2t}^n-\sum_{h,k}a_{hk}V_{2x_{h} x_{k}}^n
= \nabla_xV_{2}^n\cdot(
f_{1}\ol u_1^{(n-1)}+f_{2}\, \ol u_2^{(n-1)} +\varphi),\ in \ R_T;\nn\\
&\ol u_{1}^{(n-1)}(t,x, \nabla_{x}V_{1}^{(n-1)})\in \mathrm{Heav}(\nabla_{x} V_{1}^{(n-1)}\cdot f_{1});\nn\\
&\ol u_{2}^{(n-1)}(t,x, \nabla_{x}V_{2}^{(n-1)})\in \mathrm{Heav}(\nabla_{x}V_{2}^{(n-1)}\cdot f_{2});\nn\\
&V_i^n(0,x)=g_i(x),\ x\in\re^N.\nn
\end{align}
Conditions (A) p.34 of \cite{Ik} are satisfied and, from the
equiboundedness of $\overline u_i$,  the constants involved in this
condition  depend only on the growth
assumption \eqref{linear}, in particular are independent on $n$. Hence we can apply Theorem 2 p.41 of
\cite{Ik} and then there exists an unique weak solution
$\{V_{1}^n,V_{2}^n\}$ of problem \eqref{HHn} in $R_{T'}=(0,T')\times\re^N$ where $T'$ depends on the constant of conditions (A) p.34 of \cite{Ik}, i.e on the growth assumption \eqref{linear} thus it is independent on $n$.\\
We now prove that in our case $T'=T$. Indeed the existence result cited above
(Theorem 2 p. 41 of \cite{Ik}) is based on Theorem 3 p. 639 of  \cite{Ar}. In this Theorem the author finds
$T\leq \frac{C}{\gamma}$ where $C$ depends only on the bound of the diffusion term $\sigma$ and $\gamma$ comes from the assumption on the initial data:
\begin{equation}\label{gar}
e^{-\gamma |x|^2} g_i(x)\in L^2(\re^N).
\end{equation}
In our case since $g_i$ have polynomial growth (see assumption \eqref{gpoly}), assumption \eqref{gar} is satisfied for any $\gamma>0$.  Hence for any $T>0$ we can choose a sufficiently small $\gamma$ such that $T\leq \frac{C}{\gamma}$, thus the existence of the weak solution $\{V_{1}^n, V_{2}^n\}$ of the problem \eqref{HHn} is proved for any $T>0$.\\
Moreover, still from Theorem 2 p. 41 of \cite{Ik},   there exists a constant $\mu$ independent on $n$ such that the following estimates hold
\begin{equation}\label{stimen}
\|e^{-\mu(1+|x|^2)^{\lambda}}V_i^n\|^2_{2,\infty, R_{T}}+\|e^{-\mu(1+|x|^2)^{\lambda}}\nabla_{x}V_i^n\|^2_{2, 2, R_{T}}\leq C,
\end{equation}
where $\lambda$ is any number in $(0,1]$ and $C$ depends only on the
data i.e. is independent on $n$. At this point, following the
argument of \cite{Ar} for the proof of Theorem 3 p. 640-641, from
the weak compactness of $L^{2,2}$, up to subsequences, we have that
there exist $V_i$ and $\overline V_i$ such that
\begin{eqnarray*}
&e^{-\mu(1+|x|^2)^{\lambda}}V_i^n\to e^{-\mu(1+|x|^2)^{\lambda}}V_i\\
&e^{-\mu(1+|x|^2)^{\lambda}}\nabla_{x}V_i^n\to   e^{-\mu(1+|x|^2)^{\lambda}}\overline V_i,
\end{eqnarray*}
where the convergence is weak in $L^{2,2}(R_{T})$ and  $\lambda\in(0,1]$.\\
Let us consider now a bounded domain in $R_{T}$, $(0,T)\times B_R$ with $R$ fixed.
Following the procedure of \cite{Ar} p. 641, using the fact that
$V_{i}^n$ are weak solutions of \eqref{HHn} in the sense of definition \eqref{defweak},
we obtain that
$$V_{i}^n\to V_i,\  \nabla_{x}V_i^n\to \nabla_{x}V_i,$$
weakly in $L^{2,2}((0,T)\times B_R)$ for any $R$.
Hence $\overline V_i =\nabla_{x}V_i$ in the sense of distributions.
Note that, from \eqref{stimen}, also the limit function satisfies
\begin{equation}\label{estVi}
\|e^{-\mu(1+|x|^2)^{\lambda}}V_i\|^2_{2,2, R_{T}}+\|e^{-\mu(1+|x|^2)^{\lambda}}\nabla_{x}V_i\|^2_{2, 2, R_{T}}\leq C,
\end{equation}
and $V_i\in L^{2}[ (0,T), H_{loc}^{1,2}(R_{T})]$. Moreover from
estimates \eqref{stimen} and Lemma 3 p. 633 of \cite{Ar} we know
that
$$\|e^{-\mu(1+|x|^2)^{\lambda}}V_i\|^2_{2, \infty, R_{T}}\leq C$$ and hence
$V_i\in  L^{\infty}\left[ (0,T), L^2_{loc}(\re^N)\right]$. Finally
from Lemma 2 p. 624 of \cite{Ar} we deduce that
$$\|V_i^n\|_{2p', 2q', (0,T)\times B_R}\leq C$$
where $p'$ and $q'$ are values whose H\"{o}lder conjugates $p$ and $q$ satisfy
$\frac{N}{2p}+\frac{1}{q}\leq 1$.
Hence, up to subsequences,  we have
$V_i^n\to V_i$ weakly in the space $L^{2p', 2q'}((0,T)\times B_R).$
If we take now a test function $\Phi$ with compact support, then
from Definition \eqref{defweak}, letting $n\to \infty$  it follows that $V_i$ is a weak solution of the Cauchy problem \eqref{unbound} with the required regularity of Definition
\eqref{defweak}.
Moreover following Corollary 3.1 of \cite{Ar} the sequence $(V_i^n)_n$ converges uniformly
to $V_i$ in any compact subset of $R_T$.

Next let us show estimate \eqref{growth}. It is enough to show that for some
positive constant $C>0$, 
$$|V_i^n(t,x)|\leq
C(1+|x|^\beta),\,\,\forall (t,x)\in [0,T]\times \re^N, i=1,2.$$ 
Let
$(t,x)\in \esp$ be fixed and let $(X^{t,x}_s)_{s\in [t,T]}$ be the
solution of the following stochastic differential equation:
$$\begin{array}{l}
\xt_s=x+\int_t^s\s(r,\xt_r)dB_s, s\in [t,T] \mbox{ and }\xt_s=x \mbox{ for }s\in [0,t].
\end{array}$$
On the other hand let $\bar u_1^{n-1}$ and $\bar u_2^{n-1}$ be the stochastic processes defined by:
$\forall s\in [0,T]$,
$$\begin{array}{l}
\bar u_1^{n-1}(s)=
\ol u_{1}^{(n-1)}(s,\xt_s, \nabla_{x}V_{1}^{(n-1)}(s,\xt_s))\mbox{ and }\bar u_2^{n-1}(s)=\ol u_{2}^{(n-1)}(s,\xt_s, \nabla_{x}V_{2}^{(n-1)}(s,\xt_s)).\ea
$$
Finally let $P^{\bar u_1^{n-1},\bar u_2^{n-1}}$ be the probability on $\O$ (\cite{karatzasshreve}, pp.200) such that
$$
dP^{\bar u_1^{n-1},\bar
u_2^{n-1}}=\zeta_T\{\int_0^{.}\sigma^{-1}(s,\xt_s)\Psi(s,\xt_s,\ol
u_1^{n-1}(s),\ol u_2^{n-1}(s))dB_s\}.dP$$ where
$\Psi(s,x,u_1,u_2):=f_{1}(t, x) u_1+ f_{2}(t, x)u_{2}+
\varphi(t,x)$. Under $P^{\bar u_1^{n-1},\bar u_2^{n-1}}$, the
dynamics of $\xt$ is the following:
\begin{eqnarray*}
&&\xt_s=x+\int_t^s\Psi(s,\xt_s,\ol u_1^{n-1}(s),\ol
u_2^{n-1}(s))ds+\int_t^s\s(r,\xt_r)dB^{\bar u_1^{n-1},\bar
u_2^{n-1}}_s,\mbox{ for }s\in [t,T],\\
&&\xt_s=x \mbox{ for }s\in [0,t]
\end{eqnarray*}
where $B^{\bar u_1^{n-1},\bar u_2^{n-1}}$ is a Brownian motion under
$P^{\bar u_1^{n-1},\bar u_2^{n-1}}$.\\
 Thanks to the uniform
linear growth of $\Psi$, we have (see \cite{karatzasshreve}, p.306),
\begin{equation}\label{estx}
E^{\bar u_1^{n-1},\bar u_2^{n-1}}[\sup_{r\leq T}|\xt_r|^\beta]\le C(1+|x|^\beta)
\end{equation}where $C$ is a constant independent of $n$. On the other hand as in the proof of Theorem \ref{nashpoint} of the next section,
\begin{equation}\label{estimvn}
V_i^n(T-t,x)=E^{\bar u_1^{n-1},\bar u_2^{n-1}}[g^i(\xt_T)].
\end{equation}
Now, as $g^i$, $i=1,2$, have polynomial growth (see \eqref{gpoly}), then by \eqref{estx}
there exists a positive constant $C$ such that, for $i=1,2$, 
\begin{equation}\label{estimvn2}
|V_i^n(t,x)|\le C(1+|x|^\beta)
\end{equation}which is the claim.

Next the local H\"{o}lder continuity of $V_i$ follows from its local
boundedness \eqref{growth} together with the Interior H\"{o}lder
Continuity result recalled in Theorem C p.616 of \cite{Ar} (one can
see also the H\"{o}lder regularity result proved for linear equations in \cite{LSU}, Theorem
10.1, p.204).
\end{proof}

To prove the existence of Nash equilibria we look for solutions as
in the previous sections, i.e. that satisfy system \eqref{unbound}
almost everywhere and that belong to $W^{1,2}_q$ in any bounded subdomain
of $R_T$. Hence we have to introduce a definition of strong solution
(it is analogous to the definition of strong solution \eqref{def}
but we have dropped the boundedness in $R_T$):
\begin{definition}\label{defunb}
$(V_{1},V_{2})$ is a {\em strong} solution
of the system \eqref{unbound},  if
\begin{align}
&a)\ V_{1}(t,x),\ V_{2}(t,x) \in H^{1+\alpha}(\ol \Omega_{T})\cap
W^{1,2}_{q}(\Omega_{T}), \\& \qq \qq \qq \text{ for any bounded subdomain } \Omega_T \subset R_T, \ \alpha\in (0,1),\ q>N+2\,;\nn\\
&c)\text{ Equations}\
\eqref{unbound}\text{ hold almost everywhere in } \Omega_T \text{ and }\
V_i(0,x)= g_i(x) \text{ holds in } \Omega.\nn
\end{align}
\end{definition}
Using the same technique used to prove Theorem \eqref{existence3},
we obtain the following result which gives appropriate regularity of
$V_1$, $V_2$ to obtain the existence of Nash equilibria.
\begin{corollary}
Under assumption (A1), \eqref{linear}, \eqref{gpoly} there exists a strong solution of the Cauchy problem \eqref{unbound}.
\end{corollary}

\section{Connection with the NZSDG}
In the following result we make the connection between the solutions
of the PDEs \eqref{sistemagenerale1}-\eqref{ricontorno} (resp. \eqref{HHD}; resp. \eqref{weakversion}),  and the NZSDG. This result provides information on the features of the Nash point of the game which could be useful in several contexts, such as its numerics and simulation, etc. Note that in bounded domains it is already known (see e.g. \cite{BeFr1,Frie1}).
\begin{theorem}\label{nashpoint}
For $i=1,2$ and $(t,x)\in [0,T]\times \re^N$, let us set
$w^i(t,x)=V^i(T-t,x)$ where $(V^1,V^2)$ is a solution of the system
\eqref{sistemagenerale1}-\eqref{ricontorno} (resp. \eqref{HH}, resp. \eqref{HHD}). Then the pair of
controls $$\ba {l}(\ol u^1,\ol u^2)=(\ol u^1(s),\ol u^2(s))_{s\leq T}:=\\\qquad \qquad (\ol u^1(s,X_s,\nabla_x w^1(s,X_s),
\nabla_x w^2(s,X_s), \ol u^2(s,X_s,\nabla_x w^1(s,X_s), w^2(s,X_s))_{s\leq T}\ea $$ is Nash equilibrium for the nonzero-sum
stochastic differential game. Moreover for any $t\leq T$, $w^i(t,X_t)=J_t^i(\ol u^1,\ol u^2)$, i.e., 
$(w^1(t,X_t),w^2(t,X_t))_{t\leq T}$ are the associated Nash conditional payoffs of the game.
\end{theorem}
\begin{proof}As the functions $V^i$, $i=1,2$, belong to $H^{1+\alpha}(\ol \Omega_{T})\cap
W^{1,2}_{q}(\Omega_{T})$ for any bounded subdomain $\Omega_T \subset R_T, \ \alpha\in (0,1),\ q>N+2$ then also $w^i$, $i=1,2$ have the same regularity. Next let $R$ be fixed. Therefore one can find two sequences $(w^{1,n})_{n\ge 0}$ and $(w^{2,n})_{n\geq 0}$ of ${\cal C}^{1,2}(\esp)$ such that: for $i=1,2$,

(i) $(w^{i,n})_n$ (resp. $(\nabla_x w^{i,n})_n)$ converges uniformly to $w^i$ (resp. $\nabla_x w^i$) in $B_{R,T}$;

(ii) $(\partial_t w^{i,n})_n$ (resp.
$(D^2_{xx}w^{i,n})_n$) converges in $L^q(B_{R,T},dt\otimes dx)$ to $\partial_t w^{i}$ (resp.
$D^2_{xx}w^{i}$).
\ms

\noindent Now recall the sequence of stopping times $(\tau_R)_{R\geq 1}$ given in (\ref{eq:vlocal}) which is non-decreasing and converges to $T$ as $R\rightarrow \infty$. Next making use of It\^o's formula we obtain:
\begin{equation}\label{approxvi}\begin{array}{l}
w^{i,n}(\tau_R,X_{\tau_R})=w^{i,n}(t\wedge \tau_R,X_{t\wedge \tau_R})+\int_\tt^{\t_R}\nabla_x w^{i,n}(s,X_s)dX_s+
\int_\tt^{\t_R}{\cal L}w^{i,n}(s,X_s)ds
\end{array}
\end{equation}
where
$$
{\cal L}w^{i,n}(t,x):=\partial_tw^{i,n}(t,x)+\frac{1}{2}\sum_{i,j}a_{ij}(t,x)D^2_{x_ix_j}w^{i,n}(t,x).
$$
But under condition (\ref{uniellip}), for any $s> 0$, the random
variable $X_s$ has a density $p(0,x;s,y)dy$ which satisfies (see e.g.
\cite{Ar}, p.891 for more details)
$$\frac{1}{C_1\sqrt{2\pi s}}\exp\{-{\frac{1}{2C_1s}}\|x-y\|^2\}\le
p(0,x;s,y)\leq \frac{1}{C_2\sqrt{2\pi s}}\exp\{-{\frac{1}{2C_2s}}\|x-y\|^2\}
$$
for some constants $C_1$ and $C_2$ positive. Therefore for any $t\leq T$, we have
$$\ba {l} E[|
\int_\tt^{\t_R}{\cal L}w^{i,n}(s,X_s)ds-\\\qquad\qquad\qquad  
\int_\tt^{\t_R}
H_{i}(s,X_s,\nabla_{x}w^{i}(s,X_s), (\ol u_{1},\ol u_2)(s,X_s, \nabla_{x}w^{1}(s,X_s),\nabla_{x}w^{2}(s,X_s)))ds|]\rw_n 0.\ea $$
Going back now to (\ref{approxvi}) and take the limit w.r.t  $n$ to obtain: $\forall t\leq T$,
$$\begin{array}{l}
w^{i}(\tau_R,X_{\tau_R})=w^{i}(t\wedge \tau_R,X_{t\wedge \tau_R})+\int_\tt^{\t_R}\nabla_x w^{i}(s,X_s)dX_s\\\\\qquad\qquad  -
\int_\tt^{\t_R}H_{i}(s,X_s,\nabla_{x}w^{i}(s,X_s), (\ol u_{1},\ol u_2)(s,X_s, \nabla_{x}w^{1}(s,X_s),\nabla_{x}w^{2}(s,X_s)))ds
\end{array}
$$
which implies that
\begin{equation}\label{eq:retrowi}\begin{array}{l}
w^{i}(t\wedge \tau_R,X_{t\wedge \tau_R})=E^{\ol u_1,\ol u_2}[w^{i}(\tau_R,X_{\tau_R})\\\qq+\int_{t\wedge \tau_R}^{\t_R}h_{i}(s,X_s, (\ol u_{1},\ol u_2)(s,X_s, \nabla_{x}w^{1}(s,X_s),\nabla_{x}w^{2}(s,X_s)))ds|\cf_{t\wedge \tau_R}
].
\end{array}
\end{equation}But 
$$
E^{\ol u_1,\ol u_2}[w^{i}(\tau_R,X_{\tau_R})|\cf_{t\wedge \tau_R}]=E^{\ol u_1,\ol u_2}[w^{i}(\tau_R,X_{\tau_R})-g^i(X_T)|\cf_{t\wedge \tau_R}]+
E^{\ol u_1,\ol u_2}[g^i(X_T)|\cf_{t\wedge \tau_R}]
$$
and $(E^{\ol u_1,\ol u_2}[w^{i}(\tau_R,X_{\tau_R})-g^i(X_T)|\cf_{t\wedge \tau_R}])_{R\geq 1}$ (resp. $(E^{\ol u_1,\ol u_2}[g^i(X_T)|\cf_{t\wedge \tau_R}])_{R\geq 1}$) converges to $0$ (resp. $E^{\ol u_1,\ol u_2}[g^i(X_T)|\cf_t]$) in $L^1(dP^{\ol u_1,\ol u_2})$ as $R\rightarrow \infty$. 
Then
$$
\lim_{R\rightarrow \infty}E^{\ol u_1,\ol u_2}[w^{i}(\tau_R,X_{\tau_R})|\cf_{t\wedge \tau_R}]\overset{L^1(P^{\ol u_1,\ol u_2})}=E^{\ol u_1,\ol u_2}[g^i(X_T)|\cf_{t}] .
$$
In the same way we deal with the second term of \eqref{eq:retrowi} to obtain
$$\ba{l}\lim_{R\rightarrow \infty}E^{\ol u_1,\ol u_2}[\int_{t\wedge \tau_R}^{\t_R}h_{i}(s,X_s,\ol u^1(s),\ol u^2(s))ds|\cf_{t\wedge \tau_R}]\overset{L^1(P^{\ol u_1,\ol u_2})}=E^{\ol u_1,\ol u_2}[\int_{t}^{T}h_{i}(s,X_s,\ol u^1(s),\ol u^2(s))ds|\cf_{t}].\ea$$Therefore take the limit w.r.t $R$ in both hand-sides of \eqref{eq:retrowi} to deduce that $$w^i(t,X_t)=J_t^i(\ol u^1,\ol u^2),\,\,\forall t\leq T;\,\,i=1,2.$$
Next let us fix $i=1$ and let $u_1:=(u_{1t})_{t\leq T}$ be an admissible control for the first player. For $t\leq T$ let us set
$$
J^{1}_t(u_1,\ol u_2)=E^{u,\ol u_2}[\int_t^Th_1(s,X_s, u_{1s},\ol
u_2(s))ds+g^1(X_T)|\cf_t].
$$
Therefore by the representation property one can find a progressively measurable process $Z^{1,u_1}=(Z^{1,u_1}_t)_{t\leq T}$ such that for any $t\leq T$,
$$
J^{1}_t(u_1,\ol u_2)=g^1(X_T)+\int_t^T
H_{1}(s,X_s,\sigma^{-1}(t,X_t)^\top Z^{1,u_1}_s, u_{1s},\ol u_{2}(s))ds-\int_t^TZ^{1,u_1}_sdB_s.
$$
Taking into account of (A0), (GIC condition) it implies that $$
w^{1}(t\wedge \tau_R,X_{t\wedge \tau_R})-J^{1}_{t\wedge \tau_R}(u_1,\ol u_2)\geq E^{u_1,\ol u_2}[
w^{1}(\tau_R,X_{\tau_R})-J^{1}_{\tau_R}(u_1,\ol u_2)|\cf_{t\wedge \tau_R}].
$$
Take now the limit as $R\rw \infty$ to obtain $w^{1}(t,X_{t})-J^{1}_t(u_1,\ol u_2)\geq 0$ since $w^1(T,X_T)=g^1(X_T)=J^{1}_T(u_1,\ol u_2)$ and the probabilities $P^{u_1,\ol u_2}$ and $P$ are equivalent. Thus for any $t\leq T$, 
$J_t^1(\ol u_1,\ol u_2)\geq J_t^1(u_1,\ol u_2)$ for any admissible control $u_1$ of the first player $\pi_1$. In the same way one can show that
$J_t^2(\ol u_1,\ol u_2)\geq J_t^2(\ol u_1,u_2)$ for any admissible control $u_2$ of the second player $\pi_2$. Take now $t=0$ in the previous inequalities to deduce that 
$J_0^1(\ol u_1,\ol u_2)\geq J_0^1(u_1,\ol u_2)$ and $J_0^2(\ol u_1,\ol u_2)\geq J_0^2(\ol u_1,u_2)$ which means that $(\ol u_1,\ol u_2)$ is Nash equilibrium point for the NZSDG. Finally note that $w^i(t,X_t)=J_t^i(\ol u_1,\ol u_2)$, $i=1,2$ and then $(w^1(t,X_t),w^2(t,X_t))_{t\leq T}$ are the Nash conditional payoffs of the NZSDG.
\end{proof}
\noindent{\sc Acknowledgments }:\\
The second author is member of the INDAM-Gnampa and is partially
supported by the research project of the University of Padova
"Mean-Field Games and Nonlinear PDEs" and by the Fondazione CaRiPaRo
Project "Nonlinear Partial Differential Equations: Asymptotic
Problems and Mean-Field Games". Part of this research was done while
the second author visited Le Mans University.

\ed